\pdfoutput=1
\documentclass[reqno]{amsart}

\usepackage{amsmath,amssymb,amsthm,verbatim, csquotes,mathrsfs,stmaryrd}
\usepackage{hyperref}
\usepackage{xcolor}
\usepackage{esint}
\usepackage{mathtools}
\usepackage{graphicx}
\usepackage{float}
\usepackage{caption}
\mathtoolsset{showonlyrefs}
\usepackage{geometry}

\newtheorem{theorem}{Theorem}[section]
\newtheorem{lemma}[theorem]{Lemma}
\newtheorem{proposition}[theorem]{Proposition}
\newtheorem{definition}[theorem]{Definition}
\newtheorem{corollary}[theorem]{Corollary}
\newtheorem{remark}[theorem]{Remark}

\def\Om{\Omega}
\def\p{\partial}
\def\ep{\epsilon}
\def\de{\delta}

\def\S{{\Sigma}}
\def\<{\langle}
\def\>{\rangle}

\def\na{\nabla}

\providecommand{\abs}[1]{\lvert#1\rvert}

\providecommand{\norm}[1]{\lVert#1\rVert}

\newcommand{\mbM}{\mathbb{M}}
\newcommand{\mbN}{\mathbb{N}}

\newcommand{\mcH}{\mathcal{H}}

\newcommand{\mcL}{\mathcal{L}}

\newcommand{\mcR}{\mathcal{R}}

\newcommand{\mfc}{\mathbf{c}}
\newcommand{\mfn}{\mathbf{n}}
\newcommand{\mfq}{\mathbf{q}}
\newcommand{\mfv}{\mathbf{v}}

\newcommand{\mfR}{\mathbf{R}}
\newcommand{\mfS}{\mathbf{S}}

\newcommand{\rd}{{\rm d}}

\newcommand{\wsc}{\overset{\ast}{\rightharpoonup}}

\newcommand{\ra}{\rightarrow}

\newcommand{\eq}[1]{\begin{equation}\begin{alignedat}{2} #1 \end{alignedat}\end{equation}}

\numberwithin{equation} {section}

\begin{document}

	
\title[Compactness of capillary hypersurfaces]{Compactness of capillary hypersurfaces with mean curvature prescribed by ambient functions}
\date{\today}

\author[Zhang]{Xuwen Zhang}
	\address[X.Z]{Mathematisches Institut\\Universit\"at Freiburg\\Ernst-Zermelo-Str.1\\79104\\Freiburg\\ Germany}
\email{xuwen.zhang@math.uni-freiburg.de}

\begin{abstract}
We prove a compactness result for capillary hypersurfaces with mean curvature prescribed by ambient functions, which generalizes the results of Sch\"atzle \cite{Schatzle01} and Bellettini \cite{Bellettini23} to the capillary case.
The proof relies on extending the definition of (unoriented) curvature varifolds with capillary boundary introduced by Wang-Zhang \cite{WZ} to the context of oriented integral varifolds.
We also discuss the case when the mean curvature of the boundary is prescribed.
\

\end{abstract}

\maketitle
\tableofcontents
\section{Introduction}
In \cite{Schatzle01}, Sch\"atzle proposed and addressed the following problem:
for a sequence of hypersurfaces $M_l\subset\mfR^{n+1}$, arising as reduced boundary of sets of finite perimeter $E_l$ in $\mfR^{n+1}$, with the mean curvature of $M_l$ prescribed by an ambient function $g_l$ for each $l\in\mbN$.
If $M_l$ converges in a certain weak sense (in fact, as unoriented integral varifolds) to $M$ and $g_l$ converges in a certain sense 
to some ambient function $g$, will the prescribed mean curvature property preserved in the limit?

We call this a \textit{prescribed mean curvature problem}.
As discussed in \cite{Schatzle01}, the difficulty of solving this problem lies in the fact that, in the weak limit there might be \textit{hidden boundary}, which is due to cancellations when parts of $M_l$ meet having opposite normals.
See \cite[pp. 376]{Schatzle01} for an illustrative example, which is due to Große-Brauckmann \cite{G-B93,G-B98}.

Sch\"atlze's results show that, when $g_l$ converges to $g$ in the $W^{1,p}$-sense for $p$ suitably ranged, and $M_l$ has uniformly mass bounds, then the prescribed mean curvature property is preserved in the limit (\cite[Theorem 1.1]{Schatzle01}).
Furthermore, the situations whether there is boundary cancellation happening in the limit or not can be distinguished by the parity of density of the limiting unoriented integral $n$-varifold (\cite[Theorem 1.2]{Schatzle01}).

Recently, Bellettini \cite{Bellettini23} provides a significantly different aspect to tackle the prescribed mean curvature problem.
Precisely, as shown by Sch\"atzle, the boundary cancellation is due to the fact that parts of $M_l$ meet might have opposite normals, and is strongly related to the density of the limit weak solution.
Therefore comes the sharp observation by Bellettini, that a much more suitable weak solution to work with, is the oriented integral $n$-varifold introduced by Hutchinson \cite{Hutchinson86}, as the phenomenon described above is exactly involved in the definition of oriented integral $n$-varifolds, see \eqref{defn:V-OIV} below.

To show compactness in the context of oriented integral-$n$ varifolds with the prescribed mean curvature property preserved in the subsequential limit, instead of using Sobolev embeddings as in \cite{Schatzle01}, a special subclass of oriented integral $n$-varifolds, called \textit{oriented integral curvature varifolds}, is studied in \cite{Bellettini23}.
In this subclass, varifolds are equipped with \textit{weak second fundamental form}, and we note that this idea goes back to Hutchinson's definition of (unoriented) curvature varifolds \cite{Hutchinson86}.

With the help of some useful Geometric Measure Theory tools, for example, the criterion of second order rectifiability for integral varifolds by Menne \cite{Menne13} (see also Santilli's result for general varifolds \cite{Santilli21}), Bellettini shows that oriented integral curvature varifolds, in the co-dimension-$1$ case, posses nice properties that $C^2$-hypersurfaces are supposed to have.
By virtue of this, the prescribed mean curvature problem is solved (\cite[Theorem 1.2]{Bellettini23}).

In view of the above results, it would be interesting to ask, if assuming a hypersurface has non-empty boundary, then whether we could prescribe the mean curvature of it, and at the same time the boundary of it, in the sense that its boundary lies entirely in a given supporting hypersurface, with boundary contact angle function also prescribed by an ambient function (see \eqref{condi:capillary-angle} below).
In fact, hypersurfaces with these properties are known as \textit{capillary hypersurfaces}, which are critical points of the free energy functional and hence are natural objects to be studied.
We also use the notation \textit{$\beta$-capillary hypersurface} when we wish to emphasize that the angle prescribing function is given by $\beta$.
For a historical overview of the study of capillary surfaces, we refer to a nice survey by Finn \cite{Finn99} and also to his book \cite{Finn86}.

The main focus of the paper is to give an affirmative answer to this question, and for the problem described above, we call it the \textit{capillary prescribed mean curvature problem}.
Inspired by Bellettini's results on oriented integral curvature varifolds \cite{Bellettini23}, and by the recent work of Wang-Zhang \cite{WZ}, in which a new Neumann boundary condition for general varifolds, together with the definition of unoriented curvature varifolds with capillary boundary, is introduced.
In this paper, we will study in the co-dimension-$1$ case the oriented integral curvature varifold with capillary boundary, and use it to solve the capillary prescribed mean curvature problem.
\subsection{Capillary boundary condition}

To suitably define the capillary boundary condition for oriented curvature varifolds, we first need to introduce the following.

\begin{definition}[Oriented capillary bundle]\label{Defn:orien-bundle}
\normalfont
For a possibly unbounded domain $\Om\subset\mfR^{n+1}$ of class $C^1$ with boundary, let $\beta:S\to (0,\pi)$ be a $C^1$-function on $S$.
For any $x\in S$, we define $G^o_{n,\beta}(x)$ to be the collection of unit vectors $v\in\mfS^n$ satisfying
\eq{\label{eq:<v,nu^S>}
\abs{\left<v,\nu^S(x)\right>}
=\abs{\cos\beta(x)}.
}

By collecting all such $G^o_{n,\beta}(x)$ for $x\in S$, we obtain a subspace of $G^o_n(S)=S\times\mfS^n$, denoted by $G^o_{n,\beta}(S)$, endowed with the subspace topology.

For any $(x,v)\in G^o_{n,\beta}(S)$ we define a unit vector field
\eq{\label{defn:mfn^o}
\mfn^o(x,v)
\coloneqq\frac{v^\perp(\nu^S(x))}{\abs{v^\perp(\nu^S(x))}},
}
which is well-defined thanks to $\sin\beta>0$.
\end{definition}
As said, this definition is inspired by the unoriented version introduced in \cite{WZ}, which we record in Definition \ref{Defn:unorien-bundle} below.
The key point of the definition is that, $G^o_{n,\beta}(x)$ provides all possible choices of unit normal for any $\beta$-capillary hypersurfaces whose boundary contains the point $x$.
Fixing $(x,v)\in G^o_{n,\beta}(S)$ means that we fix a boundary point and a unit normal for the capillary hypersurface at this point, in this case
the vector $\mfn^o(x,v)$ gives exactly the inwards-pointing unit co-normal at the point $x$, associated with the choice of unit normal $v$.

With Definition \ref{Defn:orien-bundle} we can then define the measure-theoretic capillary boundary of an oriented curvature varifold to be simply a 
Radon measure on $G^o_{n,\beta}(S)$.
In fact, the precise definition is as follows.
Note however, to state it we unavoidably need to first introduce some necessary terminologies from Geometric Measure Theory, which we opt to collect in Section \ref{Sec:2}.

\begin{definition}[Oriented integral curvature varifolds with capillary boundary]\label{Defn:OICV}
\normalfont
Let $\Om\subset\mfR^{n+1}$ be a bounded domain of class $C^2$, $\beta\in C^1(S, (0,\pi))$.
Given $V$ an oriented integral $n$-varifold on $\overline\Om$, $\Gamma$ a Radon measure on $G^o_{n,\beta}(S)$.
We say that $V$ has curvature and prescribed contact angle $\beta$ with $S$ along $\Gamma$, if there exist real-valued $V$-measurable functions $W_{ia}$, for $i,a\in\{1,\ldots,n+1\}$, defined $V$-a.e., such that the following identity holds for any $\phi\in C_c^1(\mfR^{n+1}\times\mfR^{n+1})$
\eq{\label{eq:ICV-1stvariation-1}
\int\left(\left(\de_{ij}-v_iv_j\right)D_j\phi+W_{ia}D^\ast_a\phi-\left(W_{rr}v_i+W_{ri}v_r\right)\phi\right)\rd V
=-\int\phi\mfn_i^o\rd\Gamma.
}
Here $D_j$ denotes the partial derivative with respect to the variable $x_j$ in the first factor, and $D^\ast_a$ with respect to the variable $v_a$ in the second factor.
$V$ is said to have curvature in $L^p$ if $W_{ia}\in L^p(V)$ for all $i,a$.
\end{definition}
When $\Gamma=0$, this reduces to Bellettini's definition of oriented integral curvature varifolds.
One can easily check that this definition is satisfied by smooth capillary hypersurfaces, see Remark \ref{Rem:C^2-surface}.
The nice properties of the above curvature varifolds will be shown in Section \ref{Sec:3}, by virtue of which we can solve the capillary prescribed mean curvature problem.
\subsection{Main results}
Our first main result is as follows.
\begin{theorem}\label{Thm:main}
Let $\Om\subset\mfR^{n+1}$ be a bounded domain of class $C^2$, $S\coloneqq\p\Om$ with $\nu^S$ the inner unit normal to $\Om$ along $S$, $\beta\in C^1(S, (0,\pi))$.
For $l\in\mbN$, let $E_l\subset\Om$ be open subset with
$D_l\coloneqq\p E_l\cap\p\Om$
and $M_l\coloneqq\p E_l\cap\Om$ a $C^2$-hypersurface (embedded in $\mfR^{n+1}$), such that $M_l$ meets $S$ transversally along $\p M_l$ with contact angle prescribed by the function $\beta$ in the following sense: let $\nu_l$ denote the inner unit normal to $E_l$ along $\overline{M_l}$, then
\eq{\label{condi:capillary-angle}
\left<\nu_l(x),\nu^S(x)\right>=-\cos\beta(x),\quad\forall x\in\p M_l;
}
 let $g_l,g:\overline\Om\ra\mfR$ be functions in $C^0(\overline\Om)$ such that the mean curvature vector of $M_l$ is given by $g_l\nu_l$.
If
\eq{\label{condi:area-function}
\sup_{l\in\mbN}\{\mcH^n(M_l)\}<\infty;\quad
g_l\ra g\text{ in }C^0(\overline\Om);
}
for some Lebesgue measurable set $E\subset\Om$ (which will be proved to be a set of finite perimeter in $\mfR^{n+1}$ and we denote by $\p^\ast E$ its reduced boundary, by $\nu_E$ its measure-theoretic inwards pointing unit normal.
We also write $D=\p^\ast E\cap\p\Om$, which will also be proved to be a set of finite perimeter in $S$) such that,
\eq{\label{condi:set-convergence}
\chi_{E_l}\ra\chi_E\text{ in }L^1(\Om);\quad
\chi_{D_l}\ra\chi_{D}\text{ in }L^1(S);
}
for some $p>1$ fixed,
(Denote by $B^l$ the classical second fundamental form of $M_l$) there holds
\eq{\label{condi:curvature}
\sup_{l\in\mbN}\left\{\int_{M_l}\abs{B^l}^p\rd\mcH^n\right\}<\infty;
}
and (Denote by $H^l_\p$ the classical mean curvature of $\p M_l$ as a hypersurface in $S$)
\eq{\label{condi:mean-curvature}
\sup_{l\in\mbN}\left\{\int_{\p M_l}\abs{H_\p^l}\rd\mcH^{n-1}\right\}<\infty.
}
Then there exists an oriented integral $n$-varifold $V$, a Radon measure $\Gamma$ on $G^o_{n,\beta}(S)$, to which $V_l=\mfv(M_l,\nu_l,1,0)$ and $\Gamma_l=\mcH^{n-1}\llcorner\p M_l\otimes\de_{\nu_l(x)}$ subsequentially converges to, such that
\begin{enumerate}
    \item [(i)] $V$ has curvature $W_{ia}\in L^p(V)$ for $i,a\in\{1,\ldots,n+1\}$ and prescribed contact angle $\beta$ with $S$ along $\Gamma$ in the sense of Definition \ref{Defn:OICV}.
    \item [(ii)] $V=\mfv(\p^\ast E\cap\Om,\nu_E,1,0)+\mfv(\mcR,\nu,\theta,\theta)$, where $\mcR$ is $n$-rectifiable on $\overline\Om$; $\nu$ is any measurable choice of unit normal on $\mcR$; $\theta:\mcR\ra\mbN\setminus\{0\}$ is in $L^1(\mcH^n\llcorner\mcR)$; and $\mcR\subset\{x\in\overline\Om:g(x)=0\}$.
    Moreover, $\norm{V}$ is $n$-rectifiable of class $C^2$.
    \item [(iii)] There exists an $(n-1)$-rectifiable set $R_S\subset S$ and a positive integer-valued integrable Borel function $\tau:R_S\ra\mfR$ such that
    \eq{
    \Gamma=\norm{\Gamma}\otimes\Gamma^x
    =\tau\mcH^{n-1}\llcorner R_S\otimes\Gamma^x,
    }
    where $\Gamma^x$ is a Radon probability measure resulting from disintegration.
    In particular,
    \eq{\label{impli:partialD-Gamma}
    \p^\ast D\subset R_S.
    }
    Moreover, for $\Gamma$-a.e. $(x,v)$, there holds
\eq{\label{eq:bdry-contact-weaksense}
\left<\mfn^o(x,v),\nu^S(x)\right>=\sin\beta(x),
}
where $\mfn^o$ is the unit vector field defined as \eqref{defn:mfn^o}.
    \item [(iv)] The first variation takes the form
    \eq{
    \de(\mfq_\#V)=\vec H\norm{V}-\mfn\mfq_\#\Gamma,
    }
    where $\mfq$ is the projection map from the oriented Grassmannian bundle to the unoriented Grassmannian bundle, $\mfn$ is the unit vector field defined as \eqref{defn:mfn}, and
    \eq{
    \vec H(x)
    =\begin{cases}
         g(x)\nu_E(x),\quad&\text{on }(\p^\ast E\cap\Om)\setminus\mcR,\\
         0\quad&\text{on }\mcR.
    \end{cases}
    }
    
\end{enumerate}
\end{theorem}
Note that the set $\mcR$ in conclusion ($ii$) is exactly the hidden boundary.
A direct consequence is, if the nodal set of $g$ is $\mcH^n$-negligible, then there is no hidden boundary (with respect to $\mcH^n$) when taking limit.
\begin{corollary}\label{Cor:negligible-nodal-set}
Under the assumptions and notations in Theorem \ref{Thm:main}.
If
\eq{
\mcH^n(\{x\in\overline\Om:g(x)=0\})=0,
}
then the conclusions of Theorem \ref{Thm:main} hold and ($ii$), ($iv$) read as follows.
\begin{enumerate}
    \item [(ii')] $V=\mfv(\p^\ast E\cap\Om,\nu_E,1,0)$ and $\mcH^{n}\llcorner(\p^\ast E\cap\Om)$ is $n$-rectifiable of class $C^2$.
    \item [(iv')] The first variation takes the form
    \eq{
    \de(\mfq_\#V)=\vec H\mcH^n\llcorner(\p^\ast E\cap\Om)-\mfn\mfq_\#\Gamma,
    }
    where $\mfq$ is the projection map from the oriented Grassmannian bundle to the unoriented Grassmannian bundle, $\mfn$ is the unit vector field defined as \eqref{defn:mfn}, and
    \eq{
    \vec H(x)=g(x)\nu_E(x),\quad x\in\p^\ast E\cap\Om.
    }
\end{enumerate}
\end{corollary}
Observe that, as revealed by the implication \eqref{impli:partialD-Gamma}, there is hidden boundary in the limit with respect to $\mcH^{n-1}$.
On one hand,
an easy way to solve this issue is to impose the convergence of perimeters, which intuitively should forbid any chances of boundary cancellation in the limit.
In fact, we have the following.

\begin{corollary}\label{Cor:converges-perimeter}
Under the assumptions and notations in Theorem \ref{Thm:main}.
If \eqref{condi:mean-curvature} is replaced by
\eq{\label{condi:perimeter-converge}
\mcH^{n-1}(\p M_l)
=\mcH^{n-1}(\p^\ast D_l)\ra \mcH^{n-1}(\p^\ast D)\text{ as }l\ra\infty,
}
then the conclusions of Theorem \ref{Thm:main} still hold and ($iii$) improves as follows.
\begin{enumerate}
    \item [(iii')] The Radon measure $\Gamma$ takes the form
    \eq{
    \Gamma=\norm{\Gamma}\otimes\Gamma^x
    =\mcH^{n-1}\llcorner(\p^\ast D)\otimes\Gamma^x,
    }
    where $\Gamma^x$ is a Radon probability measure resulting from disintegration.
    Moreover, for $\Gamma$-a.e. $(x,v)$, there holds
\eq{
\left<\mfn^o(x,v),\nu^S(x)\right>=\sin\beta(x),
}
where $\mfn^o$ is the unit vector field defined as in \eqref{defn:mfn^o}.
\end{enumerate}
\end{corollary}
On the other hand, partly motivated by this observation we wonder, whether one could prescribe the mean curvature of the boundary of a capillary hypersurface $\p M_j$, as a hypersurface in $S$, by an ambient function $g^S_j\in C^0(S)$, which converges as continuous function to $g^S\in C^0(S)$.
We call this a \textit{capillary prescribed boundary mean curvature problem}.

To solve this problem, thanks to the co-dimension-$1$ essence of $\p M_j\subset S$, once we view $S$ as a Riemannian manifold embedded in $\mfR^{n+1}$, the prescribed mean curvature problem of $\p M_j$ in $S$ can be then tackled in the context of oriented integral curvature varifolds on Riemannian manifolds, which is completely solved by Bellettini, see \cite[Section 5]{Bellettini23}.

Combining \cite[Theorem 4]{Bellettini23} with a general compactness result which is essentially shown in the proof of Theorem \ref{Thm:main} (see Theorem \ref{Thm:general-cpt}), the \textit{capillary prescribed boundary mean curvature problem} can be then solved.
Here we point out that, to use Bellettini's results, in contrast to the uniform $L^1$-mean curvature bounds \eqref{condi:mean-curvature} in Theorem \ref{Thm:main}, we need a stronger uniform $L^q$-curvature (Second fundamental form of $\p M_l\subset S$) bounds for some $q>1$.

Our second main result is as follows.

\begin{theorem}\label{Thm:double}
Let $\Om\subset\mfR^{n+1}$ be a bounded domain of class $C^2$, $S\coloneqq\p\Om$ with $\nu^S$ the inner unit normal to $\Om$ along $S$, $\beta\in C^1(S, (0,\pi))$.
For $l\in\mbN$, let $E_l\subset\Om$ be open subset with
$D_l\coloneqq\p E_l\cap\p\Om$
and $M_l\coloneqq\p E_l\cap\Om$ a $C^2$-hypersurface (embedded in $\mfR^{n+1}$), such that $M_l$ meets $S$ transversally along $\p M_l$ with contact angle prescribed by the function $\beta$ in the following sense: let $\nu_l$ denote the inner unit normal to $E_l$ along $\overline{M_l}$, then
\eq{
\left<\nu_l(x),\nu^S(x)\right>=-\cos\beta(x),\quad\forall x\in\p M_l.
}
Let $\bar\nu_l$ denote the inwards pointing unit normal to $D_l$ along $\p M_l$, and let $g^S_l,g^S:S\ra\mfR$ be functions in $C^0(S)$ such that the mean curvature vector of $\p M_l\subset S$ is given by $g^S_l\bar\nu_l$.
If
\eq{\label{condi:area-function-bdry}
\sup_{l\in\mbN}\{\mcH^n(\p M_l)\}<\infty;\quad
g^S_l\ra g^S\text{ in }C^0(S);
}
for some $\mcH^{n-1}$-measurable set $D\subset S$ (which is in fact a set of finite perimeter in $S$ and we denote by $\p^\ast D$ its reduced boundary, and by $\nu_D$ its measure-theoretic inwards pointing unit normal) there holds
\eq{\label{condi:set-convergence-bdry}
\chi_{D_l}\ra\chi_{D}\text{ in }L^1(S);
}
for some $p>1$ fixed,
(Denote by $B^l$ the classical second fundamental form of $M_l$) there holds
\eq{\label{condi:curvature-2}
\sup_{l\in\mbN}\left\{\int_{M_l}\abs{B^l}^p\rd\mcH^n\right\}<\infty;
}
and for some $q>1$ fixed (Denote by $B^l_\p$ the classical second fundamental form of $\p M_l$ as a hypersurface in $S$)
\eq{\label{condi:bdry-curvature}
\sup_{l\in\mbN}\left\{\int_{\p M_l}\abs{B_\p^l}^q\rd\mcH^{n-1}\right\}<\infty.
}
Then there exists an oriented integral $n$-varifold $V$, an oriented integral $(n-1)$-varifold $\Gamma^S$ (on the Riemannian manifold $S$), a Radon measure $\Gamma$ on $G^o_{n,\beta}(S)$, to which $V_l=\mfv(M_l,\nu_l,1,0)$, $\Gamma^S_l=\mfv(\p M_l,\bar\nu_l,1,0)$, and $\Gamma_l=\mcH^{n-1}\llcorner\p M_l\otimes\de_{\nu_l(x)}$ subsequentially converges to, such that
\begin{enumerate}
    \item [(i)] $V$ has curvature $W_{ia}\in L^p(V)$ for $i,a\in\{1,\ldots,n+1\}$ and prescribed contact angle $\beta$ with $S$ along $\Gamma$ in the sense of Definition \ref{Defn:OICV}.
    \item [(ii)] $\Gamma^S=\mfv(\p^\ast D,\nu_D,1,0)+\mfv(\mcR_S,\bar\nu,\theta_S,\theta_S)$, where $\mcR_S$ is $(n-1)$-rectifiable on $S$; $\bar\nu$ is any measurable choice of unit normal on $\mcR_S\subset S$; $\theta_S:\mcR_S\ra\mbN\setminus\{0\}$ is in $L^1(\mcH^{n-1}\llcorner\mcR_S)$; $\mcR_S\subset\{x\in S:g^S(x)=0\}$.
    Moreover, $\norm{\Gamma^S}$ is $(n-1)$-rectifiable of class $C^2$.
    \item [(iii)]
    The Radon measure $\Gamma$ takes the form
    \eq{
    \Gamma=\norm{\Gamma}\otimes\Gamma^x=\norm{\Gamma^S}\otimes\Gamma^x,
    }
    where $\Gamma^x$ is a Radon probability measure resulting.
    In particular, for $\Gamma$-a.e. $(x,v)$, there holds
\eq{\label{eq:bdry-contact-weaksense-2}
\left<\mfn^o(x,v),\nu^S(x)\right>=\sin\beta(x),
}
where $\mfn^o$ is the unit vector field defined as \eqref{defn:mfn^o}.
    \item [(iv)]
    In terms of tangential vector field to $S$, the first variation takes the form
    \eq{
    \de(\mfq_\#\Gamma^S)
    =\vec H_\p\norm{\Gamma},
    }
    with
    \eq{
    \vec H_\p(x)
    =\begin{cases}
         g^S(x)\nu_D(x),\quad&\text{on }(\p^\ast D)\setminus\mcR_S,\\
         0\quad&\text{on }\mcR_S.
    \end{cases}
    }
\end{enumerate}
\end{theorem}
In this case,
$\mcR_S$ is the hidden boundary, and for the same reason as in Corollary \ref{Cor:negligible-nodal-set}, we could exclude the hidden boundary provided that the nodal set of $g^S$ is $\mcH^{n-1}$-negligible. 
\subsection{Outline of the paper}
In Section \ref{Sec:2} we collect the necessary ingredients from Geometric Measure Theory.
In Section \ref{Sec:3} we prove the basic properties of oriented integral curvature varifolds with capillary boundary.
In Section \ref{Sec:4} we prove the main results in order.

\subsection{Acknowledgements}
The author thanks Prof. Ernst Kuwert and Prof. Guofang Wang for many insightful discussions.
\section{Preliminaries}\label{Sec:2}
We adopt the following basic notations throughout the paper.
\begin{itemize}
    \item $\mfR^{n+1}$ with $n\geq1$ denotes the Euclidean space, with the Euclidean scalar product denoted by $\left<\cdot,\cdot\right>$ and the corresponding Levi-Civita connection denoted by $\na$.
    When considering the topology of $\mfR^{n+1}$, we denote by $\overline{E}$ the topological closure of a set $E$,
    and by $\p E$ the topological boundary of $E$.
    Let $\{e_1,\ldots,e_{n+1}\}$ denote the canonical coordinate vectors, then $\mfR^{n+1}$ is oriented by the orientation $e_1\wedge\ldots\wedge e_{n+1}$.
    \item For $k\in\mbN$, $\mcH^k$ is the $k$-dimensional Hausdorff measure on $\mfR^{n+1}$.
    \item For a Radon measure $\mu$ on $\mfR^{n+1}$:
    \begin{enumerate}
        \item ${\rm spt}\mu$ is the \textit{support} of the measure $\mu$.
        For $f:\mfR^{n+1}\ra\mfR^{n+1}$ proper, the \textit{push-forward} of $\mu$ through $f$ is the outer measure $f_\#\mu$ defined by the formula
        \eq{
        \left(f_\#\mu\right)(E)
        =\mu\left(f^{-1}(E)\right),\quad E\subset\mfR^{n+1}.
        }
        See \cite[Section 2.4]{Mag12};
        \item For $k\in\mbN$, we say that $\mu$ is \textit{$k$-rectifiable} if there exists a $k$-rectifiable set $M$ and a positive function $\tau\in L^1_{\rm loc}(M,\mcH^k)$ such that $\mu=\tau\mcH^k\llcorner M$.
        A set $M$ is \textit{$k$-rectifiable} if it can be covered, up to a $\mcH^k$-negligible set, by a countable family of $C^1$, $k$-dimensional submanifolds of $\mfR^{n+1}$.
    \end{enumerate}
    \item For a product space $X\times Y$, $\pi$ denotes the canonical projection to $X$.
    Namely, $\pi(x,y)=x$, $(x,y)\in X\times Y$.
    \item For a set of finite perimeter/Caccioppoli set $E$ (in $\mfR^{n+1}$ or in a Riemannian manifold), we denote by $\mu_E$ its \textit{Gauss-Green measure}, $\p^\ast E$ its \textit{reduced boundary}, $\nu_E$ its \textit{measure-theoretic inwards pointing unit normal}, and $P(E)$ its \textit{perimeter}.
    For basic knowledge on sets of finite perimeter, we refer to the monograph \cite{Mag12}.
    \item For a bounded domain $\Om$ of class $C^2$, denote by $h_S$ the second fundamental form of the hypersurface $S\subset\mfR^{n+1}$.
    There exists $\rho_S=\rho_S(n,\abs{h_S}_{C^0(S)})>0$, such that the \textit{normal exponential map}
\eq{\label{eq:one-sided-tubular}
S\times[0,\rho_S)\ra U^+_{\rho_S}(S),\quad(x,\rho)\mapsto x+\rho\nu^S(x)
}
is a diffeomorphism.
Here $U^+_{\rho_S}$ is the one-sided tubular neighborhood of $S$.
We use $d_S$ to denote the signed distance function with respect to $S$ such that $d_S>0$ in $\Om$, then $\abs{\na d_S}\leq1$ for $\mcL^{n+1}$-a.e. on $\mfR^{n+1}$, and $\abs{\na^2 d_S}$ is controlled on $U^+_{\frac{\rho_S}2}(S)$.
See e.g., \cite[Section 14.6]{GT01}.
\end{itemize}
The following disintegration for Radon measures is well-known, see e.g., \cite{AFP00,AGS08}.
\begin{lemma}[Disintegration]
Let $\gamma$ be a Radon measure on $X\times Y$ with compact support, where $X,Y$ are metric spaces, and denote by $\pi:X\times Y\ra X$ the projection map.
Then $\norm{\gamma}=\pi_\#\gamma$ is a Radon measure, and there is a $\norm{\gamma}$-a.e. uniquely determined family $(\gamma^x)_{x\in X}$ of Radon probability measures on $Y$, such that for any Borel function $\phi:X\times Y\mapsto[0,\infty]$ one has
\eq{
\int_{X\times Y}\phi\rd\gamma
=\int_X\int_Y\phi(x,y)\rd\gamma^x(y)\rd\norm{\gamma}(x).
}
\end{lemma}

\subsection{Oriented integral varifolds}\label{Sec:2-OIV}

We work in the co-dimension-1 case, let us first recall the definitions of oriented and unoriented Grassmannian bundle, see \cite{Hutchinson86,Bellettini23}.
\begin{itemize}
    \item The \textit{oriented Grassmannian} of $n$-planes in $\mfR^{n+1}$ is identified with $\mfS^n$ through \textit{Hodge star operator} $\star$ as follows:
    for every $v\in\mfS^n$, $\star v$ is the unit $n$-vector which spans $v^\perp$ and whose orientation is such that $\star v\wedge v$ agrees with the orientation of $\mfR^{n+1}$.
    For any $A\subset\mfR^{n+1}$, the \textit{oriented Grassmannian bundle} over $A$, denoted by $G^o_n(A)$, is then identified with $A\times\mfS^n$, which will be viewed as embedded in $A\times\mfR^{n+1}$.
    \item The \textit{unoriented Grassmannian} of (unoriented) $n$-planes, denoted by $G(n,n+1)$, is identified with the $(n+1)\times(n+1)$-matrices of orthogonal projection onto the said plane, which in turn is identified with a submanifold of $\mfR^{(n+1)^2}$.
    For any $A\subset\mfR^{n+1}$, the \textit{unoriented Grassmannian bundle} over $A$, denoted by $G_n(A)=A\times G(n,n+1)$, is then identified as a subspace in $A\times\mfR^{(n+1)^2}$.
    \item Let $\mfq:A\times\mfS^n\ra A\times \mfR^{(n+1)^2}$ denote the projection from the oriented Grassmannian bundle to the unoriented Grassmannian bundle, by the relation
    \eq{\label{defn:mfq}
    \mfq(x,v)=(x,(P_{ij})_{i,j=1,\ldots,n+1}),
    }
    where the entries of the projection matrix are given by
    \eq{
    P_{ij}
    =\de_{ij}-v_iv_j,\text{ for }i,j\in\{1,\ldots,n+1\}.
    }
\end{itemize}
\begin{definition}
\normalfont
An \textit{oriented $n$-varifold} $V$ is a Radon measure in the oriented Grassmannian bundle $\mfR^{n+1}\times\mfS^n\subset\mfR^{n+1}\times\mfR^{n+1}$.
If the weight measure of $V$ (i.e., the push-forward of $V$ under the canonical projection map $\pi:\mfR^{n+1}\times\mfS^n\ra\mfR^{n+1}$), denoted by $\norm{V}$, is supported on some subset $A\subset\mfR^{n+1}$, then $V$ is called an oriented $n$-varifold on $A$. 

An oriented $n$-varifold $V$ is called an \textit{oriented integral $n$-varifold}, if there exist an $n$-rectifiable set $R$ in $\mfR^{n+1}$, a measurable choice of orientation $\xi:R\ra\mfR^{n+1}$ such that for $\mcH^n$-a.e. $x\in R$, $\xi(x)\in\mfS^n\subset\mfR^{n+1}$ is one of the two choices of unit normal to the approximate tangent space $T_xR$, and a couple $(\theta_1,\theta_2)$ of locally integrable $\mbN$-valued functions on $R$, with $(\theta_1,\theta_2)\neq(0,0)$ for $\mcH^n$-a.e. $x\in R$, so that $V$ acts on $\phi\in C_c^0(\mfR^{n+1}\times\mfR^{n+1})$ as follows:
\eq{\label{defn:V-OIV}
V(\phi)
=\int_R\theta_1(x)\phi(x,\xi(x))+\theta_2(x)\phi(x,-\xi(x))\rd\mcH^n(x).
}
In this situation we write $V=\mfv(R,\xi,\theta_1,\theta_2)$, and it is easy to see that $\norm{V}=(\theta_1+\theta_2)\mcH^n\llcorner R$.
\end{definition}

Convergence as oriented varifolds means weak-star convergence as Radon measures in the oriented Grassmannian bundle.
\subsubsection{Induced currents}\label{Sec:2-current}
Associated to an oriented $n$-varifold $V$ there is an $n$-current $\mfc(V)$, defined by its action on any $n$-form $\eta$ having compact support in $\mfR^{n+1}$ as follows:
\eq{
\mfc(V)(\eta)
=\int\left<\star v,\eta(x)\right>\rd V(x,v).
}
If $V=(R,\xi,\theta_1,\theta_2)$ is an oriented integral $n$-varifold then $\mfc(V)$ is the current of integration on $R$ with orientation $\xi$ and multiplicity $\theta_1-\theta_2$, i.e., $\mfc(V)=\mfc(R,\xi,\theta_1-\theta_2)$.

Following the classical notations in \cite{Simon83}, we use $\p$ to denote the \textit{boundary operator} of currents, and denote by $\mbM(\cdot)$ the \textit{mass} of currents.

Note that if $V^l\ra V$ as oriented varifolds, then $\mfc(V^l)\ra\mfc(V)$ as currents, since for any $n$-form $\eta$ with compact support, the function $\left<\star v,\eta(x)\right>$ is in $C^0_c(\mfR^{n+1}\times\mfR^{n+1})$.
\subsubsection{Induced unoriented varifolds}
Associated to an oriented $n$-varifold $V$ there is an unoriented $n$-varifold, defined just as the push-forward $\mfq_\#V$, which is then a Radon measure on the unoriented Grassmannian bundle $\mfR^{n+1}\times\mfR^{(n+1)^2}$.

If $V=\mfv(R,\xi,\theta_1,\theta_2)$ is an integral oriented $n$-varifold then $\mfq_\#V$ is the unoriented integral varifold of integration on $R$ with multiplicity $\theta_1+\theta_2$.
Moreover, the weight measure $\norm{\mfq_\#V}$ is exactly $(\theta_1+\theta_2)\mcH^n\llcorner R$, and is in fact the same as $\norm{V}$.

Note that if $V^l\ra V$ as oriented varifolds, then $\mfq_\#V^l\ra\mfq_\#V$ as unoriented varifolds.
\subsubsection{Hutchinson's compactness theorem}
We are going to use the following celebrated compactness theorem for 
oriented integral varifolds by Hutchinson \cite{Hutchinson86}.
The statement here is a slightly simplified version, but sufficient for our purposes.
\begin{theorem}\label{Thm:Hutchinson-cpt}
Let $\Om\subset\mfR^{n+1}$ be a bounded domain.
Let $\{V_l\}_{l\in\mbN}$ be oriented integral $n$-varifolds on $\overline\Om$ such that
\eq{\label{eq:Hutchison-bound}
\sup_{l\in\mbN}\left\{\norm{\mfq_\#V_l}(\overline\Om)+\norm{\de(\mfq_\#V_l)}(\overline\Om)+\mbM(\p\mfc(V_l))\right\}<\infty.
}
Then there exists an oriented integral $n$-varifold $V$ on $\overline\Om$, to which $V_l$ subsequentially converges, as oriented varifolds.
\end{theorem}
\subsection{Prescribed contact angle condition}

We first recall the following from \cite{WZ}, and we restrict ourselves to the co-dimension-1 case.
\begin{definition}[Unoriented capillary bundle]\label{Defn:unorien-bundle}
\normalfont
For a possibly unbounded domain $\Om\subset\mfR^{n+1}$ of class $C^1$ with boundary $S\coloneqq\p\Om$, denote by $\nu^S$ the inwards pointing unit normal field along $S$.
Let $\beta:S\to (0,\pi)$ be a $C^1$-function on $S$.
For any $x\in S$, we define $G_{n,\beta}(x)$ to be the collection of (un-oriented) $n$-planes $P\in G(n,n+1)$ satisfying
\eq{\label{eq:P^m(nu^S)}
    \abs{P(\nu^S(x))}=\sin\beta(x),
}
where $P(u)$ denotes the orthogonal projection of the vector $u$ onto the plane $P$.

By collecting all such $G_{n,\beta}(x)$ for $x\in S$, we obtain a subspace of $G_n(S)=S\times G(n,n+1)$, denoted by $G_{n,\beta}(S)$, endowed with the subspace topology.

For any $(x,P)\in G_{n,\beta}(S)$ we define a unit vector field
\eq{\label{defn:mfn}
\mfn(x,P)
\coloneqq\frac{P(\nu^S(x))}{\abs{P(\nu^S(x))}},
}
which is well-defined thanks to $\sin\beta>0$.
One can easily check that such a vector field serves as the inwards-pointing unit co-normal for a capillary hypersurface.
\end{definition}
Enlightened by this, we introduce Definition \ref{Defn:orien-bundle}.
By virtue of \eqref{eq:<v,nu^S>} it is clear that for any $v\in G^o_{n,\beta}(x)$, there holds
\eq{\label{eq:v^perp-sinbeta}
\abs{v^\perp(\nu^S(x))}
=\sin\beta(x),
}
where $v^\perp\in G(n,n+1)$ is the orthogonal complement of $v$, and hence
\eq{\label{eq:<n^o,nu^S>=sinbeta}
\left<\mfn^o(x,v),\nu^S(x)\right>
=\sin\beta(x),\quad\forall(x,v)\in G^o_{n,\beta}(S).
}

We point out that Definition \ref{Defn:orien-bundle} is a \textit{lift} of Definition \ref{Defn:unorien-bundle} to the oriented Grassmannian bundle over $S$ in the following sense:
From $G^o_n(S)$ to $G_n(S)$ there is the canonical projection map $\mfq$, see \eqref{defn:mfq}.
Restricted to $G^o_{n,\beta}(S)$ we get
\eq{\label{eq:mfq-G^o_beta-G_beta}
\mfq(G^o_{n,\beta}(S))
=G_{n,\beta}(S),
}
thanks to \eqref{eq:v^perp-sinbeta}.
Moreover, by definitions
\eq{\label{eq:mfn^o-mfn}
\mfn^o(x,v)
=\mfn\circ\mfq(x,v),\quad\forall(x,v)\in G^o_{n,\beta}(S).
}
\section{Oriented integral curvature varifolds with capillary boundary}\label{Sec:3}
\subsection{Definition and basic properties}
We start with the following remarks on Definition \ref{Defn:OICV}.
\begin{remark}
\normalfont
If we test \eqref{eq:ICV-1stvariation-1} with $\phi\in C_c^1(\mfR^{n+1}\times\mfR^{n+1})$ such that ${\rm spt}\phi\cap{\rm spt}\norm{\Gamma}=\emptyset$, then it reads
\eq{\label{eq:ICV-1stvariation-2}
\int\left(\left(\de_{ij}-v_iv_j\right)D_j\phi+W_{ia}D^\ast_a\phi-\left(W_{rr}v_i+W_{ri}v_r\right)\phi\right)\rd V
=0.
}
\end{remark}

\begin{remark}\label{Rem:C^2-surface}
\normalfont

Let $M$ be an oriented $C^2$-embedded hypersurface in $\Om$, with a chosen unit normal $\nu$, such that $M$ meets $S$ transversally along $\p M$, with contact angle function $\beta$ in the sense of \eqref{condi:capillary-angle}.
Let $V=\mfv(M,\nu,1,0)$ be an oriented integral varifold and $\Gamma=\mcH^{n-1}\llcorner\p M\otimes\de_{\nu(x)}$ be a Radon measure on $G^o_{n,\beta}(S)$.
Computing $\int_M{\rm div}_M\phi(x,\nu(x))\rd\mcH^n(x)$ and using the tangential divergence theorem in conjunction with the fact that $\mfn^o(x,\nu(x))$ is exactly the inwards-pointing unit co-normal along $\p M\subset M$ (recall the argument subsequent to \eqref{defn:mfn^o}),
it is direct to check that $V$ has curvature $W_{ia}$ and prescribed contact angle $\beta$ with $S$ along $\Gamma$ in the sense of Definition \ref{Defn:OICV}, with $W_{ia}(x,v)
=\na_{e_i^T}\left<\nu,e_a\right>$ for $(x,v)$ such that $v=\nu(x)$ (in this way $W_{ia}$ is defined for $V$-a.e. $(x,v)$), where $T$ denotes the tangential part of a vector with respect to $M$.

\end{remark}

To discuss the preservation of prescribed mean curvature property in compactness results, we need the following properties.

\begin{proposition}\label{Prop:M_i-H_i}
Let $\Om\subset\mfR^{n+1}$ be a bounded domain of class $C^2$, $\beta\in C^1(S, (0,\pi))$.
Let $V=\mfv(R,\xi,\theta_1,\theta_2)$ be an oriented integral varifold on $\overline\Om$, $\Gamma$ be a Radon measure on $G^o_{n,\beta}(S)$.
Assume that there exist functions $M_i\in L^p(V)$ for $i\in\{1,\ldots,n+1\}$ and $p\geq1$, such that for every $\varphi\in C^1(\mfR^{n+1})$, there holds
\eq{\label{eq:lift-1st-variation}
\int\left(\left(\de_{ij}-v_iv_j\right)D_j\varphi+M_i\varphi\right)\rd V
=-\int\varphi\mfn^o_i\rd\Gamma.
}
Then the first variation takes the form 
\eq{
\de(\mfq_\#V)
=\vec H\norm{V}-\mfn\mfq_\#\Gamma,
}
with $\vec H=(H_1,\ldots,H_{n+1})$, where
\eq{\label{defn:H_i}
H_i(x)
=\frac{\theta_1(x)M_i(x,\xi(x))+\theta_2(x)M_i(x,-\xi(x))}{(\theta_1+\theta_2)(x)}
\text{ for }\norm{V}\text{-a.e. }
}
\end{proposition}
\begin{proof}
For a fixed $i$, let $\mu=(\pi\circ\mfq)_\#(M_iV)$.
Note that $\mu<<\norm{V}$ since for any $A$ such that $\norm{V}(A)=0$, it is direct to check that $\mu(A)=0$.
This implies by Radon-Nikodym theorem that
\eq{
\mu=H_i\norm{V}.
}

To proceed, we observe that for every $\varphi\in C_c^1(\mfR^{n+1})$, the first variation of $\mfq_\#V$ acting on the vector field $\varphi e_i$ is given by
\eq{
\int{\rm div}_V(\varphi e_i)\rd\norm{V}
=\int_R(T_xR)_{ij}D_j\varphi(x)(\theta_1(x)+\theta_2(x))\rd\mcH^n(x)
=\int\left(\de_{ij}-v_iv_j\right)D_j\varphi\rd V,
}
where $T_xR$ is the approximate tangent space to $R$ at $x$, and $(T_xR)_{ij}=\de_{ij}-v_iv_j$ for $v=\pm\xi(x)$, which leads to the last equality.
On the other hand, by \eqref{eq:lift-1st-variation} we have
\eq{
\int\left(\de_{ij}-v_iv_j\right)D_j\varphi\rd V
=&-\int M_i\varphi\rd V-\int\varphi\mfn^o_i\rd\Gamma.
}
Then observe that
\eq{
-\int M_i\varphi\rd V
=-\int\varphi\rd(\pi\circ\mfq)_\#(M_iV)
=-\int H_i\varphi\rd\norm{V}
=-\int\left<\vec H,\varphi e_i\right>\rd\norm{V},
}
where $\vec H=(H_1,\ldots,H_{n+1})$, and that
\eq{
-\int\varphi\mfn^o_i\rd\Gamma
=-\int\varphi\mfn_i\rd\mfq_\#\Gamma
=-\int\left<\mfn,\varphi e_i\right>\rd\mfq_\#\Gamma.
}
Hence by linearity of the first variation we obtain: for any $X\in C_c^1(\mfR^{n+1},\mfR^{n+1})$
\eq{
\de(\mfq_\#V)(X)
=-\int\left<\vec H,X\right>\rd\norm{V}-\int\left<\mfn,X\right>\rd\mfq_\#\Gamma
}
as required.
Furthermore, note that $\vec H\in L^p(\norm{V})$ since $M_i\in L^p(V)$.

Finally, note that for any $\phi\in C_c^0(\mfR^{n+1}\times\mfR^{n+1})$ there holds by definition
\eq{
M_iV(\phi)
=V(M_i\phi)
=\int_R\theta_1(x)M_i(x,\xi(x))\phi(x,\xi(x))+\theta_2(x)M_i(x,-\xi(x))\phi(x,-\xi(x))\rd\mcH^n(x),
}
and hence for $\phi(x,v)=\varphi(x)$
\eq{
\int\varphi H_i\rd\norm{V}
=\mu(\varphi)
=&\int\varphi\rd(\pi\circ\mfq)_\#(M_iV)\\
=&\int_R\left(\theta_1(x)M_i(x,\xi(x))+\theta_2(x)M_i(x,-\xi(x))\right)\varphi(x)\rd\mcH^n(x).
}
Since $\norm{V}=(\theta_1+\theta_2)\mcH^n\llcorner R$ and $(\theta_1+\theta_2)\neq0$ for $\mcH^n\llcorner R$-a.e., we thus obtain for $\norm{V}$-a.e. $x$,
\eq{\label{eq:H_i(x)}
H_i(x)
=\frac{\theta_1(x)M_i(x,\xi(x))+\theta_2(x)M_i(x,-\xi(x))}{(\theta_1+\theta_2)(x)}.
}
\end{proof}
\begin{remark}\label{Rem:even-lift}
\normalfont
In Proposition \ref{Prop:M_i-H_i}, if in addition $M_i$ are even functions (i.e., $M_i(x,v)=M_i(x,-v)$ for $V$-a.e. ($x,v$)), then we say that $M_i$ are \textit{even lift} of $H_i$ for $i\in\{1,\ldots,n+1\}$.
In this case, for those $x$ such that $\theta_1(x)\neq0$, \eqref{defn:H_i} reads
\eq{
H_i(x)
=M_i(x,\xi(x)).
}
\end{remark}

The following crucial proposition is essentially proved in \cite{Bellettini23}.
\begin{proposition}\label{Prop:W_ia-propperties}
Let $\Om\subset\mfR^{n+1}$ be a bounded domain of class $C^2$, $\beta\in C^1(S, (0,\pi))$.
Let $V=\mfv(R,\xi,\theta_1,\theta_2)$ be an oriented integral $n$-varifold on $\overline\Om$ with curvature $W_{ia}\in L^1(V)$ $(i,a\in\{1,\ldots,n+1\})$ and prescribed contact angle $\beta$ with $S$ along $\Gamma$ as in Definition \ref{Defn:OICV}.
Assume further that
$\norm{V}$ is $n$-rectifiable of class $C^2$, with $x\notin{\rm spt}\norm{\Gamma}$ for $\norm{V}$-a.e. $x$.

Then $W_{ia}$ are odd (i.e., for $V$-a.e. we have $W_{ia}(x,v)=-W_{ia}(x,-v)$) and unique (i.e., any two choices of $W_{ia}$ in \eqref{eq:ICV-1stvariation-1} must agree $V$-a.e.) functions.
Moreover,
\begin{enumerate}
    \item [($a$)] the function $-(W_{rr}v_i+W_{ri}v_r)$ appearing in \eqref{eq:ICV-1stvariation-1} and \eqref{eq:ICV-1stvariation-2} is $V$-a.e. the even lift of $H_i$, where $H_i$ is given by \eqref{defn:H_i};
    \item [($b$)] $W_{ij}v_j=0$ for $V$-a.e.;
    \item [($c$)] For $V$-a.e. we have $W_{ij}=W_{ji}$, and $W_{ri}v_r=0$, which in turn shows that the term $(-W_{rr}v_i+w_{ri}v_r)$ in \eqref{eq:ICV-1stvariation-1} and \eqref{eq:ICV-1stvariation-2} can be written as $-W_{rr}v_i$.
\end{enumerate}
\end{proposition}
\begin{proof}
By assumption, to prove the Proposition we just have to deal with those points $x\in R\setminus{\rm spt}\norm{\Gamma}$, at which we must have: for some $r(x)<<1$, $\norm{\Gamma}(B_{r(x)}(x))=0$.
Fix any such $x$, note that when testing \eqref{eq:ICV-1stvariation-1} with some $\chi_r(\cdot)=\frac1{r^n}\chi(\frac{\cdot-(x,v)}r)$, where $\chi$ is any smooth radially symmetric bump function in $\mfR^{n+1}\times\mfR^{n+1}$, we are actually testing \eqref{eq:ICV-1stvariation-2}, provided that $r$ is sufficiently small (in fact, for any $r<r(x)$).
By virtue of this observation we could directly apply \cite[Proposition 3.1]{Bellettini23} to conclude that $W_{ia}$ are odd and unique functions.
Property $(c)$ is a by-product of \cite[Proposition 3.1]{Bellettini23};
Property $(a)$ is a direct consequence of Proposition \ref{Prop:M_i-H_i} (see also Remark \ref{Rem:even-lift}) and the fact that $W_{ia}$ are odd functions;
Property $(b)$ can be obtained similarly as in \cite[Proposition 2.2]{Bellettini23}.
More precisely, by testing \eqref{eq:ICV-1stvariation-1} with $\phi=(1-\abs{v}^2)\psi$ for any $\psi\in C_c^1(\mfR^{n+1}\times\mfR^{n+1})$.
The proof is thus completed.
\end{proof}

\subsection{Unoriented integral curvature varifolds with capillary boundary}
The following definition is introduced in \cite{WZ}, which is consistent with Hutchinson's notion of curvature varifold (see \cite{Hutchinson86}) if $\Gamma=0$, consistent with Mantegazza's notion of curvature varifolds with boundary (see \cite{Mantegazza96}) in which case the boundary varifold is given by $\mfn\Gamma$, and compatible with Kuwert-M\"uler's notion of curvature varifold with orthogonal boundary (see \cite{KM22}) if we choose $\beta\equiv\frac\pi2$.
\begin{definition}\label{Defn:CV}
\normalfont
Let $\Om\subset\mfR^{n+1}$ be a bounded domain of class $C^2$, $\beta\in C^1(S, (0,\pi))$.
Let $V$ be an unoriented integral $n$-varifold on $\overline\Om$, $\Gamma$ a Radon measure on $G_{n,\beta}(S)$.
We say that $V$ has curvature and prescribed contact angle $\beta$ with $S$ along $\Gamma$, if there exist real-valued $V$-measurable functions $B_{ijk}$, for $i,j,k\in\{1,\ldots,n+1\}$, defined $V$-a.e., such that the following identity holds for any $\varphi\in C_c^1(\mfR^{n+1}\times\mfR^{(n+1)^2})$
\eq{
\int \left(P_{ij}D_j\varphi+B_{ijk}D^\ast_{jk}\varphi+B_{rir}\varphi\right)\rd V
=-\int\varphi\mfn_i\rd\Gamma.
}
Here $D_j$ denotes the partial derivative with respect to the variable $x_j$ in the first factor, and $D^\ast_{jk}$ with respect to the variable $P_{jk}$ in the second factor.

\end{definition}

With Proposition \ref{Prop:W_ia-propperties} we can show that
Definition \ref{Defn:OICV} is consistent with Definition \ref{Defn:CV} in the following sense.

\begin{proposition}\label{Prop:consistent-CV-OCV}
Let $\Om\subset\mfR^{n+1}$ be a bounded domain of class $C^2$, $\beta\in C^1(S, (0,\pi))$.
Let $V$ be an oriented integral $n$-varifold on $\overline\Om\subset\mfR^{n+1}$ with $\norm{V}$ $n$-rectifiable of class $C^2$; 
$\Gamma$ a Radon measure on $G^o_{n,\beta}(S)$ with $x\notin{\rm spt}\norm{\Gamma}$ for $\norm{V}$-a.e. $x$.
If $V$ has curvature $W_{ia}$ for $i,a\in\{1,\ldots,n+1\}$ and prescribed contact angle $\beta$ with $S$ along $\Gamma$ in the sense of Definition \ref{Defn:OICV}, then $\mfq_\#V$ has curvature $B_{ijk}$ for $i,j,k\in\{1,\ldots,n+1\}$ and prescribed contact angle $\beta$ with $S$ along $\mfq_\#\Gamma$ in the sense of Definition \ref{Defn:CV}, where $B_{ijk}$ is a $\mfq_\#V$-a.e. well-defined function, defined as $B_{ijk}\circ\mfq(x,v)\coloneqq-(W_{ik}v_j+W_{ij}v_k)(x,v)$ for $V$-a.e. $(x,v)$.
\end{proposition}
\begin{proof}
For any $\varphi\in C_c^1(\mfR^{n+1}\times\mfR^{(n+1)^2})$, we test \eqref{eq:ICV-1stvariation-1} by $\phi(x,v)=\varphi(x,P(v))$ with $P_{ij}(v)=\de_{ij}-v_iv_j$.
Note that by chain rule, $D^\ast_a\phi=-(v_j\de_{ak}+v_k\de_{aj})D^\ast_{jk}\varphi$, thus
\eq{
\int\left(\left(\de_{ij}-v_iv_j\right)D_j\varphi-\left(W_{ik}v_j+W_{ij}v_k\right)D^\ast_{jk}\varphi-\left(W_{rr}v_i+W_{ri}v_r\right)\varphi\right)\rd V
=-\int\varphi\mfn_i^o\rd\Gamma,
}
where $\varphi$ and its derivatives are evaluated at $(x,P(v))$, which is $\mfq(x,v)$ by definition.

By Proposition \ref{Prop:W_ia-propperties}, the function $-(W_{ik}v_j+W_{ij}v_k)$ is even in $v$, thus it induces a $\mfq_\#V$-a.e. well-defined function, say $B_{ijk}$, by $B_{ijk}\circ\mfq(x,v)\coloneqq-(W_{ik}v_j+W_{ij}v_k)(x,v)$ for $V$-a.e. $(x,v)$.
By definition of push-forward, taking also \eqref{eq:mfn^o-mfn} into account, we find
\eq{
\int\left(P_{ij}D_j\varphi+B_{ijk}D^\ast_{jk}\varphi+B_{rir}\varphi\right)\rd(\mfq_\#V)
=-\int\varphi\mfn_i\rd(\mfq_\#\Gamma),
}
as required.
\end{proof}
\section{Proof of the main results}\label{Sec:4}
\subsection{Mass estimates}
We continue to assume that $\Om\subset\mfR^{n+1}$ is a bounded domain of class $C^2$, $\beta\in C^1(S, (0,\pi))$.

Using standard disintegration, any Radon measure $\Gamma$ on $G^o_{n,\beta}(S)$ can be written as $\Gamma=\norm{\Gamma}\otimes\Gamma^x$, where $\norm{\Gamma}=\pi_\#\Gamma$ is the Radon measure on $S$, $\Gamma^x$ is a Radon probability measure on $G^o_{n,\beta}(x)$ for $\norm{\Gamma}$-a.e. $x$.
Note that for the push-forward $\mfq_\#\Gamma$, which is a Radon measure on $G_{n,\beta}(S)$ thanks to \eqref{eq:mfq-G^o_beta-G_beta}, we have
$\norm{\Gamma}
=\norm{\mfq_\#\Gamma}$.

\begin{lemma}\label{Lem:compare-mass}
Let $\Om\subset\mfR^{n+1}$ be a bounded domain of class $C^2$, $\beta\in C^1(S, (0,\pi))$, let $p\geq1$.
There exists a positive constant $C=C(n,p,\beta,{\rm diam}(\Om),\abs{h_S}_{C^0(S)})<\infty$, such that:
given $V$ an oriented integral $n$-varifold in $\overline\Om\subset\mfR^{n+1}$ with $\norm{V}$ $n$-rectifiable of class $C^2$; $\Gamma$ a Radon measure on $G^o_{n,\beta}(S)$ with $x\notin{\rm spt}\norm{\Gamma}$ for $\norm{V}$-a.e. $x$.
If $V$ has curvature in $L^p$ (with curvature coefficient $W_{ia}$) and prescribed contact angle $\beta$ with $S$ along $\Gamma$ in the sense of Definition \ref{Defn:OICV}, then 
\eq{
\norm{V}(\overline\Om)
\leq C\left(\norm{\Gamma}(S)+\norm{W}_{L^p(V)}^p\right),\\
\norm{\Gamma}(S)
\leq C\left(\norm{V}(\overline\Om)+\norm{W}_{L^p(V)}^p\right),
}
where $\norm{W}_{L^p(V)}^p=\left(\int\abs{W}^p\rd V\right)$, and $\abs{W}$ is the Frobenius norm of the matrix $W=(W_{ia})$.
\end{lemma}
\begin{proof}
The proof follows from \cite{KM22,WZ}.
Taking $\phi_i(x,v)=\left<x-x_0,e_i\right>$ for some $x_0\in\Om$ in \eqref{eq:ICV-1stvariation-1}, then summing over $i\in\{1,\ldots,n+1\}$ we obtain
\eq{
n\norm{V}(\overline\Om)
\leq{\rm diam}(\Om)\left(\norm{\Gamma}(S)+\norm{W}_{L^1(V)}\right).
}
By Young's inequality we have for any $p\in(1,\infty)$ and $\ep>0$
\eq{
n\norm{V}(\overline\Om)
\leq{\rm diam}(\Om)\left(\norm{\Gamma}(S)+\frac1p\ep^{-p}\norm{W}_{L^p(V)}^p+\frac{p-1}p\ep^\frac{p}{p-1}\norm{V}(\overline\Om)\right).
}
Choosing $\ep=\ep(n,p,{\rm diam}(\Om))$ sufficiently small, we could absorb the last term on the Right and get
\eq{
\norm{V}(\overline\Om)
\leq C(n,p,{\rm diam}(\Om))\left(\norm{\Gamma}(S)+\norm{W}_{L^p(V)}^p\right)
}
as required.

On the other hand, we test \eqref{eq:ICV-1stvariation-1} with $\phi_i(x,v)=l(x)\left<\na d_s(x),e_i\right>\eqqcolon l(x)\na_id_S(x)$, where $d_S$ (recall \eqref{eq:one-sided-tubular}) is the signed distance function from $S$ (positive in $\Om$) and $0\leq l\leq1$ is a cut-off function with $l_{\mid_S}=1$ and supported in $U^+_\de(S)$ for $\de=\min\{1,\frac{\rho_S}2\}$.
Clearly, such a $\phi$ is admissible in \eqref{eq:ICV-1stvariation-1}.
This yields, after summing over $i\in\{1,\ldots,n+1\}$ and thanks to \eqref{eq:<n^o,nu^S>=sinbeta}, that
\eq{
\norm{\Gamma}(S)
\leq&\frac1{\min_{x\in S}\sin\beta(x)}\int\left<\mfn^o(x,v),\nu^S(x)\right>\rd\Gamma(x,v)\\
=&\frac1{\min_{x\in S}\sin\beta(x)}\sum_{i=1}^{n+1}\int\mfn^o_i\phi_i\rd\Gamma\\
\overset{\eqref{eq:ICV-1stvariation-1}}{=}&-\frac1{\min_{x\in S}\sin\beta(x)}\sum_{i=1}^{n+1}\int\left((\de_{ij}-v_iv_j)D_j(l\na_id_S)-(W_{rr}v_i+W_{ri}v_r)l\na_id_S\right)\rd V\\
\leq&C(n,\beta,\abs{h_S}_{C^0(S)})\left(\norm{V}(\overline\Om)+\norm{W}_{L^1(V)}\right),
}
where we have used Proposition \ref{Prop:W_ia-propperties} ($c$) and Cauchy-Schwarz inequality in the last inequality.
More precisely, the last part of the sum reads
\eq{
-\int\sum_r(W_{rr})\left<v,l\na d_S\right>\rd V
=-\int{\rm tr}W\left<v,l\na d_S\right>\rd V
\leq C(n,\Om)\int\abs{W}\rd V.
}
This proves the required estimate for $p=1$.
By Young's inequality we have for any $p\in(1,\infty)$
\eq{
\norm{W}_{L^1(V)}
\leq\frac1p\ep^{-p}\norm{W}_{L^p(V)}^p+\frac{p-1}p\ep^\frac{p}{p-1}\norm{V}(\overline\Om).
}
Taking $\ep=\frac12$ for example,
the required estimate for general $p$ then follows.
\end{proof}
\subsection{Prescribed mean curvature problem}

\begin{proof}[Proof of Theorem \ref{Thm:main}]
We divide the proof into the following steps.

\noindent{\bf Step 1. We prove several facts which follow directly from the assumptions.}

For $l\in\mbN$, by assumption $E_l\subset\Om$ is an open set of finite perimeter since $M_l=\p E_l\cap\Om$ and $D_l=\p E_l\cap\p\Om$ are $C^2$-hypersurfaces.

Write $\llbracket E_l \rrbracket=\mfc(E_l,e_1\wedge\ldots\wedge e_{n+1},1)$ as the natural integral current associated to $E_l$, 
$\llbracket\p E_l\rrbracket$ as the multiplicity-$1$ $n$-current associated to $\p E_l$, such that the orientation is consistent with $\nu_{E_l}$ (the measure-theoretic inwards pointing unit normal to $\p E_l$, which is by regularity $\nu_l$ on $M_l$ and $\nu^S$ on $D_l\setminus\p M_l$),
and similarly $\llbracket\p E_l\cap\Om\rrbracket$, $\llbracket\p E_l\cap\p\Om\rrbracket$ as the multiplicity-$1$ $n$-current associated respectively to $M_l$, $D_l$.

Let $V_l=\mfv(M_l,\nu_l,1,0)$ be oriented integral $n$-varifold associated to the $C^2$-hypersurface $M_l$, $\Gamma_l=\mcH^{n-1}\llcorner\p M_l\otimes\de_{\nu_l(x)}$ be Radon measure.
By \eqref{condi:capillary-angle} and Remark \ref{Rem:C^2-surface}, $V_l$ has curvature $W_{ia}$ and prescribed contact angle $\beta$ with $S$ along $\Gamma$ in the sense of Definition \ref{Defn:OICV}, where $W^l_{ia}(x,v)=\na_{e_i^T}\left<\nu_l,e_a\right>$ for $(x,v)$ such that $v=\nu_l(x)$.
In particular, we find that
\eq{\label{ineq:curvature-bound-W}
\sup_{l\in\mbN}\left\{\int\abs{W_{ia}^l}^p\rd V_l\right\}<\infty,\quad\forall i,a\in\{1,\ldots,n+1\},
}
thanks to the uniform curvature bound \eqref{condi:curvature}.
\begin{enumerate}
    \item By \eqref{condi:area-function} and the fact that $S$ is compact, $E_l$ is an open set of finite perimeter for $l\in\mbN$, with perimeter bounded uniformly from above:
    \eq{
    \sup_{l\in\mbN}\{\mcH^n(\p E_l)\}<\infty.
    }
    By the first convergence in \eqref{condi:set-convergence} and lower semicontinuity of perimeter (see \cite[Proposition 12.15]{Mag12}), we conclude that $E\subset\Om$ is a set of finite perimeter, and $E_l$ converges as sets of finite perimeter in $\mfR^{n+1}$ to $E$.
    In particular, this implies that
    $\llbracket\p E_l\rrbracket$ converges as $n$-currents to $\llbracket\p^\ast E\rrbracket$.
    \item By the second convergence in \eqref{condi:set-convergence}, we have the corresponding convergence of currents:
    $\llbracket\p E_l\cap\p\Om\rrbracket$ converges as $n$-currents to $\llbracket\p^\ast E\cap\p\Om\rrbracket$.
    It follows from the decomposition
    \eq{
    \llbracket\p E_l\rrbracket
    =\llbracket\p E_l\cap\p\Om\rrbracket+\llbracket\p E_l\cap\Om\rrbracket,\quad
    \llbracket\p^\ast E\rrbracket
    =\llbracket\p^\ast E\cap\p\Om\rrbracket+\llbracket\p^\ast E\cap\Om\rrbracket
    }
    and the above convergences
    that
    $\mfc(V_l)=\llbracket\p E_l\cap\Om\rrbracket$ converges as currents to $\llbracket\p^\ast E\cap\Om\rrbracket$.
    \item For $l\in\mbN$, by virtue of Remark \ref{Rem:C^2-surface}, taking into account the uniform bounds on area and curvature (see \eqref{condi:area-function}, \eqref{ineq:curvature-bound-W}), we could apply Lemma \ref{Lem:compare-mass} to $V_l$ and $\Gamma_l$ then deduce that
    \eq{\label{eq:bound-partial-M_l}
    \sup_{l\in\mbN}\{\mcH^{n-1}(\p M_l)\}<\infty.
    }
    Note that $D_l$ is a closed set in $S$ with $C^2$-boundary $\p M_l$ (so of course a set of finite perimeter in $S$).
    By virtue of \eqref{eq:bound-partial-M_l}, the second convergence in \eqref{condi:set-convergence}, and again the lower semicontinuity of perimeter, we conclude that $D=\p^\ast E\cap\p\Om$ is a set of finite perimeter in $S$.
    \item For $l\in\mbN$, let $\S_l=\mcH^{n-1}\llcorner\p M_l\otimes\de_{T_x\p M_l}$ denote the unoriented (multiplicity $1$) integral $(n-1)$-varifold associated to $\p M_l$, viewed as a $C^2$-hypersurface in $S$.
    For any $C^1$-vector field $\varphi$ which is tangential to $S$, it is easy to see that
    \eq{
    \de\S_l(\varphi)
    =\int\left<\varphi,H_\p^l\right>\rd\norm{\S_l}.
    }
    By \eqref{eq:bound-partial-M_l} and \eqref{condi:mean-curvature}, we could use Allard's integral compactness theorem for unoriented varifolds (see \cite{Allard72}) to conclude that, there exists an integral $(n-1)$-varifold $\S$,
    to which $\S_l$ subsequentially converges to.
    In particular, by definition of push-forward measure this means $\norm{\S_l}=\mcH^{n-1}\llcorner\p M_l=\norm{\Gamma_l}$ subsequentially converges to $\norm{\S}$, an integer multiplicity $(n-1)$-rectifiable measure on $S$.
\end{enumerate}

\noindent{\bf Step 2. We prove the initial compactness in the context of oriented integral curvature varifolds with capillary boundary.}

Our goal is to apply Hutchinson's compactness theorem to $\{V^l\}$, to this end we need to verify \eqref{eq:Hutchison-bound}.

First, for each $l\in\mbN$, we have by construction $\norm{\mfq_\#V_l}=\mcH^n\llcorner M_l$.
Using \eqref{condi:area-function} we easily find
\eq{
\sup_{l\in\mbN}\{\norm{\mfq_\#V_l}(\overline\Om)\}<\infty.
}

Second, since $M_l$ is a capillary $C^2$-hypersurface, direct computation shows that the first variation of $\mfq_\#V_l$ is given by
\eq{
\de(\mfq_\#V_l)
=\vec H^l\mcH^n\llcorner M_l-\mfn^o(x,\nu_l(x))\mcH^{n-1}\llcorner\p M_l,
}
where $\vec H^l$ is the mean curvature of $M$, $\mfn^o(x,\nu_l(x))$ is exactly the inner unit co-normal to $\p M_l\subset M_l$ at $x\in\p M_l$, thanks to the contact angle condition \eqref{condi:capillary-angle}.
This gives
\eq{
\norm{\de(\mfq_\#V_l)}
=\abs{\vec H^l}\mcH^n\llcorner M_l+\mcH^{n-1}\llcorner\p M_l,
}
so that
we can use Cauchy-Schwarz inequality and H\"older inequality to find
\eq{
\norm{\de(\mfq_\#V_l)}(\overline\Om)
=&\int_{M_l}\abs{\vec H^l}\rd\mcH^n+\mcH^{n-1}(\p M_l)
\leq C(n)\int_{M_l}\abs{B^l}\rd\mcH^n+\mcH^{n-1}(\p M_l)\\
\leq&C(n)\left(\int_{M_l}\abs{B^l}^p\rd\mcH^n\right)^\frac1p\mcH^n(M_l)^{1-\frac1p}+\mcH^{n-1}(\p M_l).
}
%
By \eqref{condi:area-function}, \eqref{condi:curvature}, and \eqref{eq:bound-partial-M_l} we conclude,  
\eq{
\sup_{l\in\mbN}\left\{\norm{\de(\mfq_\#V_l)}(\overline\Om)\right\}<\infty.
}

Third, by construction we have $\mfc(V_l)=\mfc(M_l,\nu_l,1)$.
Thanks again to the fact that $M_l$ is a $C^2$-hypersurface with boundary $\p M_l$, by classical Stokes Theorem we know that $\mbM(\p\mfc(V_l))=\mcH^{n-1}(\p M_l)$, so that by \eqref{eq:bound-partial-M_l} we find
\eq{
\sup_{l\in\mbN}\{\mbM(\p\mfc(V_l))\}<\infty.
}
In particular, \eqref{eq:Hutchison-bound} is verified.
Applying Hutchinson's compactness theorem (Theorem \ref{Thm:Hutchinson-cpt}) we deduce, there exists an oriented integral $n$-varifold $V$ on $\overline\Om$, to which $V_l$ subsequentially converges in the sense of oriented varifolds.
It follows that $\mfc(V_l)$ converges (after passing to the said subsequence) as currents to $\mfc(V)$, as discussed in Section \ref{Sec:2-current}.
Taking Item (2) of {\bf Step 1} into account, we deduce
\eq{\label{eq:c(V)}
\mfc(V)
=\llbracket\p^\ast E\cap\Om\rrbracket,
}
and hence $V$ takes the form
\eq{
V=\mfv(\p^\ast E\cap\Om,\nu_E,1,0)+\mfv(\mcR,\nu,\theta,\theta)
}
for an $n$-rectifiable set $\mcR$ on $\overline\Om$; a measurable choice of unit normal $\nu$ on $\mcR$; and a $\mcH^n\llcorner\mcR$-integrable function $\theta:\mcR\ra\mbN\setminus\{0\}$.
Let us set $R=\p^\ast E\cap\Om\cup\mcR$ and write correspondingly $V$ as
\eq{\label{eq:V-R-xi-theta_1-theta_2}
V
=\mfv(R,\xi,\theta_1,\theta_2),
}
where $\xi$ is the measurable choice of orientation on $R$, which agrees with $\nu_E$ on $\p^\ast E\cap\Om$ and with $\nu$ on $\mcR\setminus(\p^\ast E\cap\Om)$;
$(\theta_1,\theta_2)$ equals $(1,0)$ on $(\p^\ast E\cap\Om)\setminus\mcR$, $(\theta+1,\theta)$ on $\p^\ast E\cap\Om\cap\mcR$, and $(\theta,\theta)$ on $\mcR\setminus(\p^\ast E\cap\Om)$.
This proves the first part of ($ii$).

On the other hand, concerning boundary measures $\Gamma_l$, we know by \eqref{condi:capillary-angle} that $\Gamma_l$ is a Radon measure on $G^o_{n,\beta}(S)$ for $l\in\mbN$.
Moreover, by construction $\norm{\Gamma_l}=\mcH^{n-1}\llcorner\p M_l$ and thanks again to \eqref{eq:bound-partial-M_l},
\eq{
\sup_{l\in\mbN}\{\Gamma_l(G^o_{n,\beta}(S))\}<\infty.
}
By weak-star compactness of Radon measures, there exists a Radon measure $\Gamma$ on $G^o_{n,\beta}(S)$, to which $\Gamma_l$ subsequentially converges to.
Note that 
\eq{
\norm{\Gamma_l}=\mcH^{n-1}\llcorner\p M_l
=\mcH^{n-1}\llcorner(\p^\ast D_l)
}
also subsequentially converges to $\norm{\Gamma}$ by definition of push-forward.

On one hand, by virtue of Item (4) of {\bf Step 1}, (upon extracting subsequence in the previous argument) we conclude that $\norm{\Gamma}=\norm{\S}$ is an $(n-1)$-rectifiable measure with integer multiplicity.
Namely,
there exists an $(n-1)$-rectifiable set $R_S\subset S$ and a positive integer-valued integrable Borel function $\tau:R_S\ra\mfR$ such that
\eq{
\Gamma
=\tau\mcH^{n-1}\llcorner R_S\otimes\Gamma^x,
}
which in turn implies that (since $\norm{V}$ is $n$-rectifiable)
\eq{\label{eq:norm(V)(S)=0}
x\notin{\rm spt}\norm{\Gamma}\text{ for } \norm{V}\text{-a.e. } x.
}
Note that \eqref{eq:bdry-contact-weaksense} follows from \eqref{eq:<n^o,nu^S>=sinbeta}.
($iii$) is thus proved.

On the other hand,
by Item (3) of {\bf Step 1}, we can use the lower semicontinuity of perimeter to find, for those $\rho$ such that $\norm{\Gamma}(\p B_\rho(y))=0$ (which holds for a.e. $\rho>0$ by \cite[Proposition 2.16]{Mag12})
\eq{
\mcH^{n-1}\llcorner(\p^\ast D)(B_\rho(y))
\leq&\liminf_{l\ra\infty}\mcH^{n-1}\llcorner(\p^\ast D_l)(B_\rho(y))\\
=&\liminf_{l\ra\infty}\norm{\Gamma_l}(B_\rho(y))
=\norm{\Gamma}(B_\rho(y)),
}
which shows the implication \eqref{impli:partialD-Gamma}.

\noindent{\bf Step 3. We prove subsequential compactness of $\vec\mu_l=\vec W^lV_l$ in the sense of vector-valued Radon measures, where $\vec W^l$ is the vector with $(n+1)^2$ entries $\{W_{ia}^l\}_{i,a}$.}

In this step we work with the subsequence of $l\in\mbN$ resulting from the subsequential compactness results proved in {\bf Step 2}.

Using H\"older inequality we get
\eq{
\abs{\vec \mu_l}(\mfR^{n+1}\times\mfR^{n+1})
=\int\abs{\vec W^l}\rd V^l
\leq\left(\int\abs{\vec W^l}^p\rd V^l\right)^\frac1p\left(\mcH^n(M_l))\right)^{1-\frac1{p}},
}
and hence by \eqref{ineq:curvature-bound-W}, \eqref{condi:area-function},
\eq{
\sup_{l\in\mbN}\left\{\abs{\vec \mu_l}(\mfR^{n+1}\times\mfR^{n+1})\right\}<\infty.
}

By weak-star compactness (see e.g., \cite[Corollary 4.34]{Mag12}), there exists a vector-valued Radon measure $\vec\mu$, to which $\vec\mu_l$ subsequentially (after extracting a further subsequence) converges to.
By virtue of {\bf Step 2}, $V^l$ subsequentially converges to $V$ as Radon measures.

Then we argue as in the proof of \cite[Proposition 4.30]{Mag12}, see also \cite{Bellettini23}.
Consider bounded open set $A\subset\mfR^{n+1}\times\mfR^{n+1}$, and let $A_t=\{x\in A:{\rm dist}(x,\p A)>t\}$ for $t>0$.
Let $\phi\in C_c^1(A,[0,1])$ be such that $\chi_{A_t}\leq\phi$, then we have (put $p'=\frac{p}{p-1}$, $\Lambda=\sup_{l\in\mbN}\left\{\int\abs{\vec W^l}^p\rd V^l\right\}$ thanks to \eqref{ineq:curvature-bound-W})
\eq{
\abs{\vec\mu}(A_t)
\leq&\liminf_{l\ra\infty}\abs{\vec\mu_l}(A_t)
\leq\liminf_{l\ra\infty}\int\phi\abs{\vec W^l}\rd V_l
\leq\liminf_{l\ra\infty}\left(\int\abs{\vec W^l}^p\rd V^l\right)^\frac1p\left(\int\phi^{p'}\rd V^l\right)^{\frac1{p'}}\\
\leq&\Lambda^\frac1p\liminf_{l\ra\infty}\left(\int\phi^{p'}\rd V^l\right)^{\frac1{p'}}
\leq \Lambda^\frac1p\left(\int\phi^{p'}\rd V\right)^\frac1{p'}
\leq\Lambda^\frac1pV(A)^\frac1{p'},
}
and hence we get $\abs{\vec\mu}(A)\leq\Lambda^\frac1pV(A)^\frac1{p'}$ for any Borel set $A$,
implying that ${\vec\mu}$ is absolutely continuous with respect to $V$ and, by virtue of \cite[Corollary 5.11]{Mag12}, is of the type $\vec\mu=\vec WV$ for $\vec W\in L^1(V)$, with entries denoted by $W_{ia}$, $i,a\in\{1,\ldots,n+1\}$.

To show that $\vec W\in L^p(V)$, we just have to apply \cite[Theorem 9.6]{Mantegazza96} for the subsequential convergences $V_{l}\ra V$ and $\vec W^lV_l\wsc\vec WV$ shown above, with the convex and lower semicontinuous function $f$ therein chosen specifically as
$f(\xi)=\abs{\xi}^p, \forall\xi\in\mfR^{n+1}$.

Note that for $l\in\mbN$, by construction $V_l,\Gamma_l$ satisfies Definition \ref{Defn:OICV}, and hence
the identity \eqref{eq:ICV-1stvariation-1} holds for any $\phi\in C_c^1(\mfR^{n+1}\times\mfR^{n+1})$ and each $l$:
\eq{\label{eq:1st-variation-V_l}
\int\left(\left(\de_{ij}-v_iv_j\right)D_j\phi+W^l_{ia}D^\ast_a\phi-\left(W^l_{rr}v_i+W^l_{ri}v_r\right)\phi\right)\rd V_l
=-\int\phi\mfn_i^o\rd\Gamma_l.
}
Using the fact that $\vec W^lV_l$ and $\Gamma_l$ subsequentially converges to $\vec WV$ and $\Gamma$ respectively, we can then pass the above identity to the limit and obtain
\eq{\label{eq:variation-V-1}
\int\left(\left(\de_{ij}-v_iv_j\right)D_j\phi+W_{ia}D^\ast_a\phi-\left(W_{rr}v_i+W_{ri}v_r\right)\phi\right)\rd V
=-\int\phi\mfn_i^o\rd\Gamma,
}
which proves that the oriented integral $n$-varifold $V$ has curvature $W_{ia}$ and prescribed contact angle $\beta$ with $S$ along $\Gamma$ in the sense of Definition \ref{Defn:OICV}.
($i$) is thus proved.

\noindent{\bf Step 4. We conclude the proof by using the prescribed mean curvature property.}

We first recall,
for $l\in\mbN$,
since $M_l$ is a $C^2$-hypersurface, we have by Remark \ref{Rem:C^2-surface} that $W_{ia}^l(x,v)=\na_{e_i^T}\left<\nu_l,e_a\right>$ for $(x,v)$ such that $v=\nu_l(x)$, which is defined for all points on the support of $V_l$.
It follows that $-\sum_rW^l_{rr}=-{\rm tr}W^l$ is the mean curvature of $M_l$ and is by assumption prescribed by $g_l$, and $\sum_rW_{ri}v_r=0$ for $(x,v)$ such that $v=\nu_l(x)$.

Thus for test function $\phi(x,v)$ \eqref{eq:1st-variation-V_l} reads
\eq{
\int\left(\left(\de_{ij}-v_iv_j\right)D_j\phi
+W^l_{ia}D^\ast_a\phi
+g_lv_i\phi\right)\rd V_l
=-\int\phi\mfn_i^o\rd\Gamma_l.
}
By \eqref{condi:area-function} we have $g_lv_i\ra gv_i$ in $C^0(\overline\Om\times\mfR^{n+1})$, using {\bf Step 3} we can then pass the above identity to the subsequential limit and obtain
\eq{\label{eq:variation-V-2}
\int\left(\left(\de_{ij}-v_iv_j\right)D_j\phi
+W_{ia}D^\ast_a\phi
+gv_i\phi\right)\rd V
=-\int\phi\mfn_i^o\rd\Gamma.
}
Comparing \eqref{eq:variation-V-2} with \eqref{eq:variation-V-1}, we obtain
\eq{
-\left(W_{rr}v_i+W_{ri}v_r\right)
=gv_i,\quad V\text{-a.e.}
}
Moreover, for $\phi(x,v)=\varphi(x)$, \eqref{eq:variation-V-2} reads
\eq{
\int\left(\left(\de_{ij}-v_iv_j\right)D_j\varphi+gv_i\varphi\right)\rd V
=-\int\varphi\mfn_i^o\rd\Gamma,
}
we can then use Proposition \ref{Prop:M_i-H_i} to conclude that
\eq{
\de(\mfq_\#V)
=\vec H\norm{V}-\mfn\mfq_\#\Gamma,
}
with $\vec H=(H_1,\ldots,H_{n+1})\in L^1(V)$ where
\eq{
H_i(x)
=\frac{\theta_1(x)M_i(x,\xi(x))+\theta_2(x)M_i(x,-\xi(x))}{(\theta_1+\theta_2)(x)}
}
for $M_i(x,v)=gv_i\in L^1(V)$ (in fact, $L^\infty(V)$) and $\xi,(\theta_1,\theta_2)$ as in \eqref{eq:V-R-xi-theta_1-theta_2}, therefore
\begin{align}
\vec H(x)
=\frac{\theta_1(x)g(x)\xi(x)-\theta_2(x)g(x)\xi(x)}{(\theta_1+\theta_2)(x)}\\
=
\begin{cases}
    g(x)\nu_E(x),\quad&\text{on }(\p^\ast E\cap\Om)\setminus\mcR,\\
    \frac{g(x)\nu_E(x)}{2\theta(x)+1},\quad&\text{on }\p^\ast E\cap\Om\cap\mcR,\label{eq:vecH-expression}\\
    0,\quad&\text{on }\mcR\setminus(\p^\ast E\cap\Om).
\end{cases}
\end{align}

In particular, this shows that the first variation $\de(\mfq_\#V)$ is of the required form (the first part of ($iii$)), with
$\norm{\de(\mfq_\#V)}=\abs{\vec H}\norm{V}+\norm{\mfq_\#\Gamma}$ a Radon measure.
By
\cite[Theorem 1]{Menne13}, we conclude that $\norm{V}$ is $n$-rectifiable of class $C^2$, which finishes the proof of ($ii$).

Recall \eqref{eq:V-R-xi-theta_1-theta_2}, \eqref{eq:norm(V)(S)=0}, and the fact that $V,\Gamma$ satisfies Definition \ref{Defn:OICV}, we can now apply Proposition \ref{Prop:W_ia-propperties} ($a$) and Remark \ref{Rem:even-lift} to find that
\eq{
H_i(x)
=&M_i(x,\xi(x))
=g(x)\left<\nu_E(x),e_i\right>
\text{ for }\norm{V}\text{-a.e. } x\in\p^\ast E\cap\Om,\\
H_i(x)
=&M_i(x,\xi(x))
=g(x)\left<\nu(x),e_i\right>\text{ for }\norm{V}\text{-a.e. } x\in\mcR\setminus(\p^\ast E\cap\Om),
}
which, in conjunction with the expressions in \eqref{eq:vecH-expression}, shows that $(\theta_1(x),\theta_2(x))=(1,0)$ (or equivalently, $\theta(x)=0$) for $\norm{V}$-a.e. $x\in\{x\in\overline\Om:g(x)\neq0\}$.
Thus, upon removing a set of $\mcH^n$-measure zero, we could assume that $\mcR\subset\{x\in\overline\Om:g(x)=0\}$.
This in turn shows that the expression of $\vec H$ is of the form
\eq{
\vec H(x)
=
\begin{cases}
    g(x)\nu_E(x),\quad&\text{on }(\p^\ast E\cap\Om)\setminus\mcR,\\
    0,\quad&\text{on }\mcR,
\end{cases}
}
as required, which finishes the proof of ($iv$),
and completes the proof.
\end{proof}

\begin{remark}
\normalfont
As shown in {\bf Step 1} in the proof of Theorem \ref{Thm:main}, condition \eqref{condi:set-convergence} is mainly used to show that $\mfc(V_l)=\llbracket\p E_l\cap\Om\rrbracket$ converges as $n$-currents to $\llbracket\p^\ast E\cap\Om\rrbracket$.
In this regard, the set convergences in \eqref{condi:set-convergence} can be replaced by the convergence of sets of finite perimeter under the so-called \textit{$\mathbf{F}$-distance} studied in \cite{LZZ24}, see in particular \cite[Lemma 2.2]{LZZ24}.
\end{remark}

The following general compactness theorem for $C^2$-hypersurfaces with capillary boundary is essentially proved above.
We record it since it may be considered of independent interest.

\begin{theorem}\label{Thm:general-cpt}
Let $\Om\subset\mfR^{n+1}$ be a bounded domain of class $C^2$, $S\coloneqq\p\Om$ with $\nu^S$ the inner unit normal to $\Om$ along $S$, $\beta\in C^1(S, (0,\pi))$.
For $l\in\mbN$, let $E_l\subset\Om$ be open subset with
$M_l\coloneqq\p E_l\cap\Om$ a $C^2$-hypersurface (embedded in $\mfR^{n+1}$), such that $M_l$ meets $\p\Om$ transversally along $\p M_l$ with contact angle prescribed by the function $\beta$ in the following sense: let $\nu_l$ denote the inner unit normal to $E_l$ along $\overline{M_l}$, then
\eq{
\left<\nu_l(x),\nu^S(x)\right>=-\cos\beta(x),\quad\forall x\in\p M_l;
}
and (Denote by $B^l$ the classical second fundamental form of $M_l$) for some $p>1$ fixed, there holds
\eq{
\sup_{l\in\mbN}\left\{\mcH^n(M_l)+\int_{M_l}\abs{B^l}^p\rd\mcH^n\right\}<\infty.
}
Then there exists an oriented integral $n$-varifold $V$, a Radon measure $\Gamma$ on $G^o_{n,\beta}(S)$, to which $V_l=\mfv(M_l,\nu_l,1,0)$ and $\Gamma_l=\mcH^{n-1}\llcorner\p M_l\otimes\de_{\nu_l(x)}$ subsequentially converges to, such that
$V$ has curvature $W_{ia}\in L^p(V)$ for $i,a\in\{1,\ldots,n+1\}$ and prescribed contact angle $\beta$ with $S$ along $\Gamma$ in the sense of Definition \ref{Defn:OICV}.
\end{theorem}

As said, Corollary \ref{Cor:negligible-nodal-set} follows directly from Theorem \ref{Thm:main}.
Now we prove the last Corollary.
\begin{proof}[Proof of Corollary \ref{Cor:converges-perimeter}]
We just need to show ($iii'$) when we replace \eqref{condi:mean-curvature} by \eqref{condi:perimeter-converge}, because the rest of the proof is exactly the same as that of Theorem \ref{Thm:main}.

As shown in Item ($3$) of {\bf Step 1} in the proof of Theorem \ref{Thm:main}, by lower semicontinuity of perimeter,
$D_l$ converges as set of finite perimeter to $D$, meaning that (thanks to De Giorgi's structure theorem) the Gauss Green measure $\mu_{D_l}=\nu_{D_l}\mcH^{n-1}\llcorner(\p^\ast D_l)$ weak star converges to the Gauss Green measure $\mu_D=\nu_{D}\mcH^{n-1}\llcorner(\p^\ast D)$.

By assumption \eqref{condi:perimeter-converge}, we can use \cite[Proposition 4.30 ($ii$)]{Mag12} to Gauss-Green measures and conclude that
$\mcH^{n-1}\llcorner(\p M_l)$ (by construction this is exactly $\norm{\Gamma_l}$), which is by regularity $\mcH^{n-1}\llcorner(\p^\ast D_l)$, weak star converges to $\mcH^{n-1}\llcorner(\p^\ast D)$.
This surpasses Item ($4$) of {\bf Step 1} in the proof of Theorem \ref{Thm:main}.
Using this fact, we deduce in {\bf Step 2} that $\norm{\Gamma}$, to which $\norm{\Gamma_l}$ subsequentially converges to, must be $\mcH^{n-1}\llcorner(\p^\ast D)$.
($iii'$) is thus proved.

\end{proof}

\subsection{Prescribed boundary mean curvature problem}
Before we give the proof of Theorem \ref{Thm:double}, we point out that, though we use the same notation $\mfv(\cdot,\cdot,\cdot,\cdot)$ for $V_l$ and $\Gamma_l^S$, the meanings of them are different.
The former stands for oriented integral $n$-varifolds on $\mfR^{n+1}$, defined in Section \ref{Sec:2-OIV}.
The later stands for oriented integral ($n-1$)-varifolds on the $n$-dimensional Riemannian manifold $S$, which is embedded in $\mfR^{n+1}$, whose precise definition can be found in \cite[Section 5]{Bellettini23}.
\begin{proof}[Proof of Theorem \ref{Thm:double}]

For $l\in\mbN$, by virtue of Remark \ref{Rem:C^2-surface}, taking into account the uniform bounds on area and curvature (see \eqref{condi:area-function-bdry}, \eqref{ineq:curvature-bound-W}), we could apply Lemma \ref{Lem:compare-mass} to $V_l$ and $\Gamma_l$ then deduce that
    \eq{\label{eq:bound-M_l}
    \sup_{l\in\mbN}\{\mcH^n(M_l)\}<\infty.
    }
    
Note also that $D_l$ is a closed set in $S$ with $C^2$-boundary $\p M_l$ (so of course a set of finite perimeter in $S$).
By virtue of \eqref{condi:area-function-bdry}, the convergence \eqref{condi:set-convergence-bdry}, and the lower semicontinuity of perimeter, we conclude that $D=\p^\ast E\cap\p\Om$ is a set of finite perimeter in $S$.

To proceed, thanks to \eqref{eq:bound-M_l} and \eqref{condi:curvature-2}, we could use Theorem \ref{Thm:general-cpt} to find that there exists an oriented integral $n$-varifold $V$, a Radon measure $\Gamma$ on $G^o_{n,\beta}(S)$, to which $V_l,\Gamma_l$ subsequentially converges to respectively, such that $V$ has curvature $W_{ia}\in L^p(V)$ for $i,a\in\{1,\ldots,n+1\}$ and prescribed contact angle $\beta$ along $S$ with $\Gamma$, which proves ($i$).
Note that $\norm{\Gamma_l}$ also subsequentially converges to $\norm{\Gamma}$ by definition of push-forward measure.

On the other hand,
for any $l\in\mbN$,
by construction $\norm{\Gamma_l^S}=\mcH^{n-1}\llcorner(\p M_l)$, therefore
by virtue of \eqref{condi:area-function-bdry}, \eqref{condi:bdry-curvature}, the fact that the $C^2$-hypersurface $\p D_l=\p M_l\subset S$ has mean curvature prescribed by $g^S_l\bar\nu_l$, and that $D_l$ converges in $L^1(S)$ to the set of finite perimeter $D$, we can thus use \cite[Theorem 4]{Bellettini23} (in particular, the Riemannian version of \cite[Theorem 2]{Bellettini23}) to deduce the existence of the oriented integral $(n-1)$-varifolds $\Gamma^S$ (on the Riemannian manifold $S$), to which $\Gamma_l^S$ subsequentially converges to, such that
\begin{itemize}
    \item $\Gamma^S=\mfv(\p^\ast D,\nu_D,1,0)+\mfv(\mcR_S,\bar\nu,\theta_S,\theta_S)$, where $\mcR_S$ is $(n-1)$-rectifiable on $S$; $\bar\nu$ is any measurable choice of unit normal on $\mcR_S\subset S$; $\theta_S:\mcR_S\ra\mbN\setminus\{0\}$ is in $L^1(\mcH^{n-1}\llcorner\mcR_S)$; $\mcR_S\subset\{x\in S:g^S(x)=0\}$.
    $\norm{\Gamma^S}$ is ($n-1$)-rectifiable of class $C^2$.
    \item In terms of tangential vector fields to $S$, the first variation takes the form
    \eq{
    \de(\mfq_\#\Gamma^S)
    =\vec H_\p\norm{\Gamma},
    }
    with
    \eq{
    \vec H_\p(x)
    =\begin{cases}
         g^S(x)\nu_D(x),\quad&\text{on }(\p^\ast D)\setminus\mcR_S,\\
         0\quad&\text{on }\mcR_S.
    \end{cases}
    }
\end{itemize}
In particular, ($ii$) and ($iv$) are proved.

Finally,
by definition of push-forward, $\norm{\Gamma^S_l}$ subsequentially converges to $\norm{\Gamma^S}$.
On the other hand, as shown above, $\norm{\Gamma_l}$ subsequentially converges to $\norm{\Gamma}$.
Combining these two convergence results and note that by construction, $\norm{\Gamma_l^S}=\mcH^{n-1}\llcorner(\p M_l)=\norm{\Gamma_l}$,
we conclude $\norm{\Gamma}=\norm{\Gamma^S}$ as required.
Since $\Gamma$ is a Radon measure on the bundle $G^o_{n,\beta}(S)$, 
\eqref{eq:bdry-contact-weaksense-2} follows immediately from \eqref{eq:<n^o,nu^S>=sinbeta}, which proves ($iii$) and completes the proof.
\end{proof}
\begin{remark}
\normalfont
In fact, the first part of the assumptions in Theorems \ref{Thm:main}, \ref{Thm:double} can be replaced by the following much more general version:

Let $\Om,\{\Om_l\}_{l\in\mbN}$ be bounded domains of class $C^2$, $S_l\coloneqq\p\Om_l$ with $\nu^{S_l}$ the inner unit normal to $\Om_l$ along $S_l$, and
$\beta_l\in C^1(S_l,(0,\pi))$, $\beta\in C^1\in C^1(S,(0,\pi))$ be functions such that
\eq{\label{condi:Om_l,beta_l}
\overline\Om_l\ra\overline\Om\text{ in }C^2\text{-topology}\text{ and }
\beta_l\ra\beta\text{ in }C^1\text{-topology}.
}
For $l\in\mbN$,
let $E_l\subset\Om_l$ be open subset with $D_l\coloneqq\p E_l\cap\p\Om_l$ and $M_l\coloneqq\p E_l\cap\Om_l$ a $C^2$-hypersurface such that $M_l$ meets $S_l$ transversally along $\p M_l$ with contact angle prescribed by the function $\beta_l$.
Then, under the same assumptions \eqref{condi:capillary-angle}-\eqref{condi:mean-curvature}, the conclusions of Theorem \ref{Thm:main} still hold.

To see this, we have to show the following two necessary modifications for the proof of Theorems \ref{Thm:main}, \ref{Thm:double} to work in this case.
First, note that by \eqref{condi:Om_l,beta_l}, we could assume that $\{S_l\}_{l\in\mbN}$ have uniformly bounded geometry, namely, $\{\abs{h_{S_l}}_{C^0(S_l)}\}_{l\in\mbN}$ is uniformly bounded.
This means, when we apply Lemma \ref{Lem:compare-mass} to each $V_l$, the constants $C_l=C_l(n,p,\beta,{\rm diam}(\Om_l),\abs{h_{S_l}}_{C^0(S_l)})$ appearing on the RHS of the estimates are uniformly bounded, and hence the estimate 
\eqref{eq:bound-partial-M_l} in the proof of Theorem \ref{Thm:main} also the estimate \eqref{eq:bound-M_l} in the proof of Theorem \ref{Thm:double} still hold in this case.
Second, by \eqref{condi:Om_l,beta_l} we have the convergences
\eq{
G^o_n(S_l)\ra G^o_n(S),\quad
G^o_{n,\beta_l}(S_l)
\ra G^o_{n,\beta}(S)\text{ as }l\ra\infty,
}
which ensures that the Radon measure $\Gamma$, as the subsequential limit of $\Gamma_l$, is supported on $G^o_{n,\beta}(S)$.
The rest of the proof are exactly the same.
\end{remark}
\begin{remark}
\normalfont
In Theorem \ref{Thm:double}, to show capillary compactness (i.e., conclusion ($i$)), there is another weak solution that we can associate the regular capillary hypersurface $M_l$ with, namely, \textit{varifolds with capillary boundary} introduced in \cite{WZ}, which are un-oriented varifolds with mean curvature but not full curvature.

In that case, we could use the general integral compactness theorem for varifolds with capillary boundary established in \cite{WZ} to obtain the desired result, and we shall replace \eqref{condi:curvature-2} by the assumption that for some $p>1$,
\eq{
\sup_{l\in\mbN}\left\{\mcH^n(M_l)+\int_{M_l}\abs{\vec H_l}^p\rd\mcH^n\right\}
<\infty.
}
Here we point out that the uniform bounds on $\{\mcH^n(M_l)\}_{l\in\mbN}$ is not really a condition since it can be inferred from the uniform bounds on $\{\mcH^{n-1}(\p M_l)\}_{l\in\mbN}$ and $\{\int_{M_l}\abs{\vec H_l}^p\rd\mcH^n\}_{l\in\mbN}$.
More precisely, for each $l\in\mbN$, because $M_l$ is a regular capillary hypersurface, the following first variation formula holds: $\forall\varphi\in C^1_c(\mfR^{n+1};\mfR^{n+1})$,
\eq{
\int_{M_l}{\rm div}_{M_l}\varphi(x)\rd\mcH^n(x)
=\int_{M_l}\left<\vec H_l(x),\varphi(x)\right>\rd\mcH^n(x)-\int_{\p M_l}\left<\mu_l(x),\varphi(x)\right>\rd\mcH^{n-1}(x),
}
where $\mu_l$ is the inwards pointing unit co-normal of $M_l$ along $\p M_l$, satisfying the capillary boundary condition $\left<\mu_l(x),\nu^S(x)\right>=\sin\beta(x),\quad\forall x\in\p M_l$.
Thus one can argue as Lemma \ref{Lem:compare-mass} to show that 
\eq{
\mcH^n(M_l)
\leq C(n,p,\beta,{\rm diam}(\Om),\abs{h_S}_{C^0(S)})\left(\mcH^{n-1}(\p M_l)+\norm{\vec H_l}_{L^p(M_l)}^p\right).
}
\end{remark}


\bibliography{BibTemplate.bib}
\bibliographystyle{amsplain}

\end{document}